\documentclass[12pt]{article}
\usepackage{graphicx}
\usepackage{amsmath}
\usepackage{amsthm}
\usepackage{amssymb,amsfonts}
\usepackage[T1]{fontenc}
\usepackage[backend=biber, maxnames=10,maxbibnames=99,style=alphabetic]{biblatex}
\usepackage{hyperref}
\usepackage{latexsym}
\usepackage{mathrsfs}
\usepackage{todonotes}
\usepackage{upgreek}
\usepackage{rotating}
\usepackage{nicefrac}
\usepackage{epsfig}
\usepackage{stmaryrd}
\usepackage{setspace}
\usepackage{enumerate}
\usepackage{bbm,ifpdf,tikz}
\usepackage{mathtools}
\usepackage{tikz}
\usetikzlibrary{cd}

\renewbibmacro{in:}{}

\newtheorem{theorem}{Theorem}[section]

\newtheorem{lemma}[theorem]{Lemma}
\newtheorem{proposition}[theorem]{Proposition}
\newtheorem{corollary}[theorem]{Corollary}
\newtheorem{conjecture}[theorem]{Conjecture}

{
\theoremstyle{definition}
\newtheorem{definition}[theorem]{Definition}
\newtheorem{example}[theorem]{Example}

\newtheorem{remark}[theorem]{Remark}

}

\newcommand{\excise}[1]{}

\renewcommand{\dim}{\operatorname{dim}}
\newcommand{\Ann}{\operatorname{Ann}}

\renewcommand{\and}{\qquad\text{and}\qquad}

\newcommand{\A}{\mathbf{A}}
\newcommand{\B}{\mathbf{B}}

\newcommand{\M}{\mathbf{M}}

\newcommand{\Gbar}{\overline{G}}
\newcommand{\Mbar}{\overline{M}}
\newcommand{\Mboldbar}{\overline{\mathbf{M}}}

\newcommand{\Q}{\mathbb{Q}}
\newcommand{\N}{\mathbb{N}}

\newcommand{\Sn}{\mathfrak{S}}
\DeclareMathOperator{\FI}{FI}

\renewcommand{\phi}{\varphi}

\title{Independence numbers in certain families of highly symmetric graphs}
\author{David Guan\footnote{DG was supported by a Gibbons fellowship from Bowdoin College} and Eric Ramos \footnote{ER was supported by NSF grant DMS-2137628}}

\begin{document}

\maketitle

{\small
\begin{quote}
\noindent {\em Abstract.}
FI-graphs were introduced by the second author and White in \cite{RW} to capture the idea of a family of nested graphs, each member of which is acted on by a progressively larger symmetric group. That work was built on the newly minted foundations of representation stability theory and FI-modules. Examples of such families include the complete graphs and the Kneser and Johnson graphs, among many others. While it was shown in the originating work how various counting invariants in these families behave very regularly, not much has thus far been proven about the behaviors of the typical extremal graph theoretic invariants such as their independence and clique numbers. In this paper we provide a conjecture on the growth of the independence and clique numbers in these families, and prove this conjecture in one case. We also provide computer code that generates experimental evidence in many other cases. All of this work falls into a growing trend in representation stability theory that displays the regular behaviors of a number of extremal invariants that arise when one looks at FI-algebras and modules.
\end{quote} }

\section{Introduction}

\subsection{The setup}
For the remainder of this work, a \textbf{graph} will always refer to a finite simple graph. In the work \cite{RW}, The second author and White introduced what they call an $\textbf{FI-graph}$ (see Definition \ref{FIgraphdef}). For the purposes of this introduction, one may think of an FI-graph as a sequence of nested graphs,
\[
G_0 \subseteq G_1 \subseteq G_2 \subseteq \ldots,
\]
satisfying the following properties:
\begin{itemize}
\item For each $n \geq 0$, the symmetric group $\Sn_n$ acts by graph homomorphisms on $G_n$
\item For all $n$ these actions are compatible with one another, in the sense that if one acts on the subgraph $G_n \subseteq G_{n+1}$ by the subgroup $\Sn_n \leq \Sn_{n+1}$ of permutations that do not move $n+1$, then the action agrees with that of the group $S_n$ on the graph $G_n$
\item For any $n \gg 0$, and any vertex $v$ of $G_{n+1}$, $v$ is some $\Sn_{n+1}$-translate of a vertex of $G_{n}\subseteq G_{n+1}$
\item For any $n \gg 0$, the transposition $(n+1,n+2)$ acts trivially on the subgraph $G_n \subseteq G_{n+2}$
\end{itemize}
barring the final condition listed above, which is there for purely technical reasons, one should think of an $\FI$-graph as being a sequence of nested graphs, indexed by the natural numbers, each acted on by the corresponding symmetric group in such a way that these actions are compatible with one another.

The first example of such an object is the collection of complete graphs $G_n = K_n$. One may also consider the collection of complete bipartite graphs $K_{n,a}$, with $a$ fixed, the collection of Kneser graphs $KG_{n,r}$, with $r$ fixed, and many others. In the original work \cite{RW}, it is examined how these families of graphs interact with the theories of representation stability \cite{CF} and Church-Ellenberg-Farb $\FI$-modules \cite{CEF} In particular, it is proven that a number of counting invariants are in agreement with polynomials. For instance, one has the following.

\begin{theorem} \label{thmofsubgraph}
Let $G_\bullet$ be an FI-graph, and let $H$ be any fixed graph. Then for $n \gg 0$, the function
\[
n \mapsto \text{ the number of copies of $H$ that appear as subgraphs of $G_n$}
\]
agrees with a polynomial.
\end{theorem}

This theorem can be seen somewhat easily in the case of $G_n = K_n$. Indeed, as one example, the number of triangles in $K_n$ is precisely equal to the polynomial $\binom{n}{3}$. One should also note that as a consequence of the above theorem, one necessarily has that the total number of vertices of $G_n$ must be in agreement with a polynomial whenever $n$ is large enough.

One direction that the originating work \cite{RW} does not touch as much upon is the question of \emph{extremal} invariants of the graph $G_n$ as a function of $n$. While it is shown in \cite{RW} that the maximal and minimal vertex degrees of $G_n$ agree with polynomials in $n$, nothing is said about, for instance, the \textbf{independence numbers}, $\alpha(G_n)$, of these graphs. In this work, we propose the following conjecture.

\begin{conjecture} \label{mainconj}
Let $G_{\bullet}$ be an $\FI$-graph. Then the generating function
\[
\sum_{n \geq 0} \alpha(G_n)t^n,
\]
is rational. In particular, the function
\[
n \mapsto \alpha(G_n)
\]
is in agreement with a quasi-polynomial for all $n \gg 0$.
\end{conjecture}

\begin{remark}
It is not hard to show that if one starts with an $\FI$-graph $G_\bullet$ and replaces each $G_n$ with its complement, then the resulting sequence of graphs is once again an $\FI$-graph. It follows that the above conjecture implies the exact same rationality conclusion about the \textbf{clique number} of the graph. That is to say, the largest clique that one can find in $G_n$.
\end{remark}

Coming back to the case of the complete graph, one certainly has $\alpha(G_n) = 1$ for all $n$. On the other hand, it is the subject of the famous Erd\"os-Ko-Rado Theorem that for the Kneser graphs $G_n = KG_{n,r}$, one has $\alpha(G_n) = \binom{n-1}{r-1}$ whenever $n \gg 0$. As a third example, if one takes $G_n$ to be the complement of the Kneser graph $KG_{n,2}$ (i.e. the Johnson Graph $J_{n,2}$), then it is not hard to show that $\alpha(G_n) = \lfloor \frac{n}{2} \rfloor$.

This paper is broken into two parts. To start, we prove the following special case of our primary conjecture.

\begin{theorem} \label{mainthm}
Let $G_\bullet$ be an $\FI$-graph, and assume that the number of vertices of $G_n$ is equal to a linear function of $n$ for all $n \gg 0$. Then there exists constants $A,B$ such that
\[
\alpha(G_n) = An + B.
\]
\end{theorem}

While this theorem might be proven in a purely combinatorial manner, we instead opt to prove it in an almost entirely algebraic fashion. This approach will require us to apply much of the state of the art of $\FI$-modules as well as $\FI$-algebras. We intentionally choose this approach as it suggests, in our opinion, what is likely the correct approach to prove the general conjecture. We will see that the combinatorics of $\FI$-graphs become drastically more complicated as soon as one allows their vertices to grow faster than linear. While this is true on the algebraic side as well, the algebra has the benefit of a robust underlying theory that might be able to contend with this growing complexity. We will touch upon how exactly our techniques generalize to the higher growth cases at the end of section \ref{proofofmainthm}.

Following these theoretical advancements, we spend the remainder of the paper providing experimental evidence for our primary conjecture. To accomplish this, we prove a classification theorem (See Section \ref{quadclass}) for all $\FI$-graphs whose vertex sets grow quadratically. This classification theorem allows us to perfectly encode the data of all the of graphs which comprise an $\FI$-graph into a finite amount of data. This reduction to the finite allows us to efficiently compute the independence numbers of many $\FI$-graphs for very large values of $n$, thereby providing evidence for our main conjecture. One may think of our approach here therefore being in line with the philosophy of \emph{quantitative} representation stability, as in \cite{DL,MMPR}. In other words, we use the machinery of representation stability to reduce an infinite problem to a finite one, and then use computer experimentation to explicitly contend with the newly finite problem.

\subsection{Our Main conjecture in the context of representation stability}

In order to paint as full a picture as possible, we spend a bit of time examining our main conjecture from the perspective of trends in representation stability theory.

One of the early realizations in the context of representation stability, as well as the theory of $\FI$-modules, is that the dimensions of their graded pieces eventually agree with a polynomial. Since these early results, however, it has become clear that there are higher levels of this behaviors which we do not yet have adequate explanation for. More specifically, it has been found both theoretically and experimentally that similar ``$\FI$-type" objects often have very regular behaviors in their \emph{extremal} invariants.

For instance, it is proven in \cite{N} that a finitely generated $\FI$-module over the $\FI$-algebra $A_n = \Q[x_1,\ldots,x_n]$ will have krull dimension that is linear as a function in $n$. It is also discussed in that paper how it has been conjectured that the projective dimensions of these modules will also be linear in $n$. Other work also suggests that ideals in this ring must have linearly growing regularity in $n$ \cite{HHQ}. More recently, Knudsen, Miller, and Tosteson \cite{KMT} proved that the 
dimensions of finite codimension homology groups of configuration spaces of manifolds have quasi-polynomial behavior.

In the work of \cite{N}, the relevant proof strategy is to encode the Hilbert polynomials of these modules as generating functions for certain regular languages. The paper \cite{KMT} approaches its subject material through homological stability of Lie Algebras. Both of these approaches are quite different, and do not suggest a single underlying mechanism to explain why their respective extremal invariants behave so regularly. our approach in this work will ultimately most closely resemble that of \cite{N}. The results of this work, as well as its main conjecture, therefore continues these trends.

\subsection*{Acknowledgements}
The first author was supported by a Gibbons Fellowship from the Bowdoin College. The second author was supported by NSF grant DMS-2137628.

\section{Background}

\subsection{Graph theory and edge ideals}

Though the graphs we use in this paper are all finite simple graphs, we should make the definition more explicit:
\begin{definition}
A (finite, simple) \textbf{graph} is an ordered pair $G = \bigl(V(G), E(G)\bigr)$, where $V(G)$ is a finite set of \textbf{vertices}, and $E(G)$ is a finite set of pairs of vertices called the \textbf{edges} of $G$. Note that by virtue of $E(G)$ being a set of pairs, we are implicitly disallowing both loops and multiedges in our graphs. 

A \textbf{homomorphism} between two graphs $G,G'$ is a map between their vertex sets that sends edges to edges.
\end{definition}

\begin{example}
Let $G = K_n$ be a complete graph of n vertices. Then the vertex set $V(G) = \{v_i$ | $i \in \mathbb{N}, 1 \leq i \leq n\}$ and the edge set $E(G) = \bigl\{\{v_i, v_j\}$ | $i, j \in \mathbb{N}, 1 \leq i < j \leq n\bigr\}$.
\end{example}

\begin{example}
Let $G = KG_{n,2}$ be a Kneser graph whose vertices are labelled by pairs of distinct natural numbers not greater than n. Then $V(G) = \{v_{i,j}$ | $i, j \in \mathbb{N}, 1 \leq i < j \leq n\}$ and $E(G) = \bigl\{\{v_{i,j}, v_{k,l}\}$ | $1 \leq i < j \leq n, 1 \leq k < l \leq n, \{i, j\} \cap \{k, l\} = \emptyset\bigr\}$.
\end{example}

\begin{example}
Let $G = \bigsqcup_{i=1}^{2} KG_{n,2}$ be a disjoint union of two Kneser graphs in previous example. Then we can set $V(G) = \{v_{o,i,j}$ | $o \in \{1,2\}, i, j \in \N, 1 \leq i < j \leq n\}$ where this extra $o$ represents which Kneser graph is this vertex from. Consequently, $E(G) = \bigl\{\{v_{o,i,j}, v_{o,k,l}\}$ | $o \in \{1,2\}, 1 \leq i < j \leq n, 1 \leq k < l \leq n,$ $\{i, j\} \cap \{k, l\} = \emptyset\bigr\}$.
\end{example}

Notice that "$i$", "$i,j$" and "$o,i,j$" each represents a label of a vertex in each of the three examples above. 

The primary graph theoretic invariant of interest in this paper is what is known as the independence number of the graph.

\begin{definition}
    Given a graph $G = (V(G),E(G))$, the \textbf{independence number} $\alpha(G)$ of $G$ is the size of the largest collection of vertices that are pair-wise not adjacent.
\end{definition}

The independence number is an extremely classical invariant that has seen interest from a number of different perspectives and fields. For instance, it is known that the computation of this invariant is strongly NP-complete \cite{GJ}. We will need to understand the independence number from an algebraic perspective.

\begin{definition}\label{polynomialringdef}
Let $G = \bigl(V(G), E(G)\bigr)$ be a graph. The \textbf{polynomial ring with respect to $G$} is the polynomial ring
\[
R_G = \mathbb{Q}[x_{L_1}, x_{L_2}, ..., x_{L_n}]
\] 
where each $x_{L_i}$ corresponds to a $v_{L_i} \in V(G)$ where $L_i$ is a label; in other words, each vertex in $G$ represents a variable of $R_G$.
\end{definition}

For the purposes of this work, the above algebra is relevant due to the presence of a particular ideal that we define now.

\begin{definition}
Let $G = (V(G), E(G))$ be a graph and $R_G$ be the polynomial ring with respect to $G$. The \textbf{edge ideal of $R_G$} is the ideal of $R_G$ defined by
\[
I_G = \bigl(x_{L_i}x_{L_j} \text{ } | \text{ } \{v_{L_i}, v_{L_j}\} \in E(G)\bigr).
\] 
In other words, $I_G$ is generated by square-free monomials of degree $2$ that represent connected pairs of vertices.
\end{definition}

One immediately observes from this definition that a square-free monomial is \emph{not} included in the edge idea if and only if every pair of vertices that appears in the monomial are not connected by an edge. This observation is the key insight for the following classical theorem.

\begin{theorem}\cite[Corollary 1.15]{MS}\label{thmdimR/I}
Let $G$ be a graph, and let $R_G$ and $I_G$ be the polynomial ring and edge ideal of $G$, respectively. Then the independence number of $G$ is equal to the Krull dimension of the quotient ring $(R/I)_G := R_G/I_G$. That is to say, it is equal to the longest proper chain of prime ideals in $(R/I)_G$.
\end{theorem}

\begin{example}
Let $G = K_5$. Then $R_G = \mathbb{Q}[x_1, x_2, x_3, x_4, x_5]$ and,
\[
I_G = (x_1x_2, x_1x_3, x_1x_4, x_1x_5, x_2x_3, x_2x_4, x_2x_5, x_3x_4, x_3x_5, x_4x_5).
\]
The only prime ideals of the quotient ring $R_G/I_G$ are those prime ideals of $R_G$ which contain $I_G$. In particular, any prime ideal of the quotient ring must contain at least four of the five variables by virtue of it needing to contain all of the necessary products. It therefore follows that the longest (proper) chain of prime ideals in the quotient has length 1. Clearly the independence number of the complete graph $K_5$ is also 1.
\end{example}

\begin{example}
Let $G = KG_{4,2}$. Then $R_G = \mathbb{Q}[x_{1,2}, x_{1,3}, x_{1,4}, x_{2,3}, x_{2,4}, x_{3,4}]$ and $I_G = (x_{1,2}x_{3,4},$ $x_{1,3}x_{2,4}, x_{1,4}x_{2,3})$. In other words, $G$ is a disjoint union of 3 edges. Graph theoretically it immediately follows that the independence number of $G$ is 3.

Algebraically, any prime ideal containing the edge ideal must contain at least three of the six variables. It follows that the Krull dimension in this case is three, as expected.
\end{example}

\subsection{$\FI$-graphs and $\FI$-modules}

This final background section serves to provide formal definitions of $\FI$-graphs, $\FI$-modules, etc., and how they are related to each other. 

\begin{definition}\label{FIgraphdef}
The category whose objects are all the finite sets $[n] := \{1,...,n\}$ and whose morphisms are set injections from $[m]$ to $[n]$ for any $m \leq n$ is called $\textbf{FI}$.

An \textbf{$\FI$-graph} is a functor $G_\bullet$ from the category $\FI$ to the category \textbf{Graph} of graphs and graph homomorphisms. We will often use $G_f$ to denote the graph homomorphism induced by the injection $f$.
\end{definition}

Intuitively, an $\FI$-graph is a collection of graphs $G_n$, each acted on by the corresponding symmetric group $\Sn_n$, which are compatible with one another according to the homomorphisms induced from set injections. Of course, without more restrictions, one quickly finds themselves with examples of $\FI$-graphs that are entirely unwieldy to work with.

\begin{example}
Let $G_0 \subseteq G_1 \subseteq G_2 \subseteq \ldots$ be any chain of graphs thought of as acted on by the corresponding symmetric group trivially. Then one may define an $\FI$-graph using the above by setting $G_f$ to be the inclusion given by the chain.
\end{example}

This example illustrates that without imposing more conditions, one can define $\FI$-graphs that almost entirely skirt all of the extra algebraic structure that we hope to impose. To overcome this issue we consider the following.

\begin{definition}
An $\FI$-graph is called \textbf{finitely generated} if for every $n \gg 0$ and every vertex $v \in V(G_{n+1})$, there is a vertex $v' \in V(G_n)$ and an injection $f: [n] \hookrightarrow [n+1]$ such that $G_f(v') = v$. 
\end{definition}

\textbf{For the remainder of this work, every $\FI$-graph will be finitely generated.}

Collections of complete graphs, complete bipartite graphs, and Kneser graphs are all classical examples of finitely generated $\FI$-graphs. One of the most significant properties of $\FI$-graphs is that they are particularly well suited for enumeration questions.

\begin{theorem} \cite[Theorem B]{RW}
Let $G_\bullet$ be an FI-graph, and let $H$ be any fixed graph. Then for $n \gg 0$, the function
\[
n \mapsto \text{ the number of copies of $H$ that appear as subgraphs of $G_n$}
\]
agrees with a polynomial. In particular, the number of vertices in an $\FI$-graph is in eventual agreement with a polynomial.
\end{theorem}

The last part of the previous theorem now gives us a means of stratifying the class of $\FI$-graphs in terms of the growth rate of their vertex sets. This leads to the following definition.

\begin{definition}\label{defVertexLinearFIGraph}
Let $G_\bullet$ be an $\FI$-graph. We say $G_\bullet$ is \textbf{vertex-linear} if for all $n \gg 0$, there are integers $A$, $B$ such that 
\[
|V(G_n)| = An + B.
\] 
\end{definition}

\begin{example}\label{exampleKnesergraph}
The collection of Kneser graphs $G_\bullet = KG_{\bullet,r}$ has $|V(G_n)| = \binom{n}{r}$, which is a polynomial of degree $r$. Therefore the Kneser graph is \textbf{not} vertex-linear for all $r \geq 2$. When $r=1$, $KG_{\bullet,1} = K_\bullet$ is the complete graph and is vertex-linear ($|V(K_n)| = n$).
\end{example}

\begin{example}
The collection of complete bipartite graphs $G_\bullet = K_{\bullet, \bullet}$ has $|V(G_n)| = 2n$, which is also vertex-linear.
\end{example}

\begin{example}
Let $(H,v)$ be an arbitrary graph with a selected vertex $v$, and let $G_n = \wedge{i=1}^{n}H$ be $n$-times wedge product of $H$ with itself. Then $G_\bullet$ is a vertex-linear $\FI$-graph with $|V(G_n)| = n \cdot (|V(H)|-1) + 1$.
\end{example}

The notion of \emph{edge}-linear $\FI$-graphs was first introduced in \cite{R}. In that work all such $\FI$-graphs were classified. One nice property of vertex-linear $\FI$-graphs is that there is also relatively simple classification theorem for them; in other words, for $n$ sufficiently large, all $G_n$ can be derived from a classification graph. The classification theorem for vertex-linear $\FI$-graphs will be included when we discuss the classification theorem for the slightly-complicated \textbf{vertex-quadratic} $\FI$-graphs in section \ref{generalcase}. 

The algebraic approach to \textbf{Theorem \ref{mainthm}} relies on the following algebraic structure: 

\begin{definition}
An \textbf{$\FI$-module over $\Q$} is a functor $\M_\bullet$ from the category $\FI$ to the category of vector spaces over $\Q$. Use $M_f$ to denote the linear transformation induced by the injection $f$. We say $M$ is \textbf{finitely generated} if there is a finite set $\{w_1,...,w_t$ | $w_i \in \M_{k_i}$ for some $k_i \in \N\}$ such that for any $n$, every element $m \in \M_n$ can be written as a finite sum
\[
m = \sum{a_iM_{f_i}(m_i)}
\]
where $a_i \in \Q$, $m_i = w_j$ for some $j$, and $f_i$ is some injection from $[k_j]$ to $[n]$.
\end{definition}

The following example shows a close connection between $\FI$-modules and $\FI$-graphs:

\begin{example}
Let $G_\bullet$ be any $\FI$-graph. Then one can obtain a (finitely generated) $\FI$-module $\M_\bullet$ where each $\M_n$ is a vector space over $\Q$ with basis indexed by the vertices of $G_n$ and transition maps send elements from basis to basis in the obvious fashion.
\end{example}

There are some nice counting invariants related to $\FI$-modules and $\FI$-graphs; for example, \cite[Theorem 1.5]{CEF} indicates that for any finitely generated $\FI$-module $\M_\bullet$, the $\Q$-dimension of $\M_n$ eventually agrees with a polynomial $p(n)$. Using the techniques in \cite[Theorem 4.1]{RW} one can obtain Theorem \ref{thmofsubgraph}.

While the above counting variants look relatively simple, the extremal invariants, such as the independence number $\alpha(G_n)$, are far from being the same. For instance, one can easily see that these Invariants cannot all be described by polynomials. We have also discussed that the independence number of the complement of the collection of Kneser graphs $KG_{n,2}$; in this case, $\alpha(G_n) = \lfloor \frac{n}{2} \rfloor$, which is clearly not a polynomial. However, it turns out that the independence number of any vertex-linear $\FI$-graph will indeed grow like a linear polynomial, as elaborated in the next chapter.

\section{Proof of Theorem \ref{mainthm}} \label{proofofmainthm}

The $\FI$-modules introduced in the previous chapter are collections of $\Q$-modules. The $\FI$-module we will use in the proof relies on a more sophisticated structure called an $\FI$-algebra. 

\begin{definition}
An $\FI$-algebra over $\Q$ is a functor $\A_\bullet$ from the category $\FI$ to the category of commutative $\Q$-algebras. Use $A_f$ to denote the $\Q$-algebra homomorphism induced by the injection $f$.
\end{definition}

\begin{example}
Let $\A_n = \Q$ for each n, and let each $A_f$ to be an identity map. Then $\A_\bullet$ itself is an $\FI$-algebra over $\Q$. 
\end{example}

\begin{example}\label{algebraexample}
Let $\B_n = \Q[x_1,...,x_n]$ be a polynomial ring over $\Q$ with $n$ variables for each n, and let each $A_f$ to be the transition map that permutes all $x_i$'s in the obvious way. Then $\B_\bullet$ is an $\FI$-algebra over $\Q$. 
\end{example}

\textbf{Throughout the remainder of the chapter,} we always use $\B_\bullet$ to represent the $\FI$-algebra in the above example.

\begin{definition}
Let $\A_\bullet$ be an $\FI$-algebra, and let any $d \in \N$. We define a \textbf{tensor power of $\A_\bullet$}, denoted $(\A^{{\otimes}d})_\bullet$, such that
\[
(\A^{{\otimes}d})_n := (\A_n)^{{\otimes}d}
\]
for each n. 
\end{definition}

\begin{example}\label{tensorexample}
It is not hard to see that $(\A^{{\otimes}d})_\bullet$ defined above is an $\FI$-algebra for any $d \in \N$. In particular, $(\B^{{\otimes}d})_\bullet$ can be view as a collection of polynomial rings over $\Q$
\[
(\B^{{\otimes}d})_n = \Q[x_{i,j} | 1 \leq i \leq d, 1 \leq j \leq n]
\]
where each $A_f$ is the transition map that fixes each $i'$ and permutes all $x_{i',j}$ in the obvious way.
\end{example}

\begin{definition}\label{FImoduleoverFIalgebradef}
Let $\A_\bullet$ be an $\FI$-algebra. An $\FI$-module over $\A_\bullet$ is an $\FI$-module $\M_\bullet$, satisfying the following: 

\begin{itemize}
\item Each $\M_n$ is an $\A_n$-module;
\item Use $M_f$ to denote the $\A_k$-module homomorphism for any injection $f$ from $[k]$ to $[n]$, such that for any $m \in \M_k$ and $a \in \A_k$, 
\[
M_f(a \cdot m) = A_f(a) \cdot M_f(m) \in \M_n.
\]
    
\end{itemize}
We say $\M_\bullet$ is \textbf{finitely generated} if there is a finite set $H \subset \bigsqcup_{i=1}^{\infty} \M_i$ such that for any element $m$ in any $\M_i$, it can be written as a finite sum
\[
m = \sum_{j}a_jM_{f_j}(h_j)
\]
where $h_j$'s are (not necessarily distinct) elements in $H$, $f_j$'s are some transition maps to $[i]$, and $a_j$'s are elements in the $\Q$-algebra $\A_i$.

\end{definition}

\begin{example}\label{exampletensorB}
Let $\M_n = \B_n$ for each n, and let each $M_f$ to be the transition map that permutes all $x_i$'s in the obvious way, just as $A_f$ in Example \ref{algebraexample}. Then $\M_\bullet$ is an $\FI$-module over $\B_\bullet$. More generally, let $\M_n = (\B^{{\otimes}d})_n$, and each $M_f$ to behave the same way as $A_f$ in Example \ref{tensorexample}, then $\M_\bullet$ is an $\FI$-module over $(\B^{{\otimes}d})_\bullet$.
\end{example}

It is known that a finitely generated $\FI$-module $\M_\bullet$ over $(\B^{{\otimes}d})_\bullet$ has a surprising linear property on the Krull dimension of each $\M_n$, which was proven by constructing a Hilbert series on (vector space) dimensions of graded pieces of each $\M_n$.

\begin{theorem}\cite[Theorem 5.16]{N}\label{linearkrulldimthm}
Let $\M_\bullet$ be a finitely generated $\FI$-module over $(\B^{{\otimes}d})_\bullet$. Then there are integer constants $A$, $B$ such that for each $n \gg 0$, the Krull dimension
\[
\dim \M_n = An + B.
\]
\end{theorem}

The next goal is to find a way to convert any vertex-linear $\FI$-graph to a well-defined $\FI$-module. 

\begin{theorem}\cite[Theorem A]{RSW}\label{theoremforlinearFIset}
Let $G_\bullet$ be a finitely generated vertex-linear $\FI$-graph with $|V(G_n)| = An+C$ for all $n \gg 0$, where $A,C \in \N$. Then there exists a number $d \in \N$ such that for every $n \geq d$, the vertex set $V(G_n)$ is isomorphic to a disjoint union of $A$ number of copies of the $\Sn_n$ action on $[n]$ along with $C$ number of $\Sn_n$-invariant vertices, also called "singletons". We say that $G_\bullet$ has vertex-linear stable degree $\leq d$ if $d$ is the smallest number satisfying the condition above.
\end{theorem}

Notice that every injection from $[k]$ to $[l]$ can be view as a composition of the standard embedding from $[k]$ to $[l]$ and a permutation of $[l]$. Since an $\FI$-graph is a functor between categories and therefore respects function compositions, every $k,l \geq d$, there is a bijective correspondence between copies of $[k]$'s in $V(G_k)$ and $[l]$'s in $V(G_l)$ (as well as the singletons in both vertex sets). The structure of $V(G_n)$ enables us to construct the polynomial ring with respect to $G_n$ (see definition \ref{polynomialringdef}) for each $n \geq d$:
\[
R_{G_n} = \Q[x_{1,1},...,x_{1,n},x_{2,1},...,x_{2,n},...,x_{A,1},...,x_{A,n},y_1,y_2...,y_C].
\]

Before we construct an $\FI$-module for this collection of $R_{G_n}$'s, we should slightly modify the $\FI$-graph $G_\bullet$ so that the new $\FI$-graph $\Gbar_\bullet$ is vertex-linear since the very beginning, and every $V(\Gbar_n)$ satisfies the condition described in the above theorem. 

\begin{definition}\label{defGbar}
Let $G_\bullet$ be a vertex-linear $\FI$-graph with vertex-linear stable degree $\leq d$. We define $\Gbar_\bullet$ as an $\FI$-graph where $\Gbar_n = G_n$ for every $n \geq d$, and $\Gbar_n$ is an empty graph whose vertex set $V(G_n)$ satisfies the condition in Theorem \ref{theoremforlinearFIset} for $n < d$. The use of empty graphs ensures that $\Gbar_\bullet$ satisfies the "graph homomorphism" property of an $\FI$-graph. In particular, $\Gbar_\bullet$ is a vertex-linear $\FI$-graph with stable degree $\leq 1$.
\end{definition}

\begin{lemma}\label{lemmatensorBnmodule}
Let $G_n$ be a vertex-linear $\FI$-graph. Using the $\Gbar_n$ defined above, let $\M_n = R_{\Gbar_n}$ for each $n \in \N$. Then each $\M_n$ is a $(\B^{{\otimes}(A+C)})_n$-module.
\end{lemma}

\begin{proof}
Recall that $(\B^{{\otimes}(A+C)})_n$ can be written as 
\[
(\B^{{\otimes}(A+C)})_n = \Q[x_{1,1},...,x_{1,n},x_{2,1},...,x_{2,n},...,x_{A,1},...,x_{A,n},y_{1,1},...,y_{1,n},...,y_{C,1},...,y_{C,n}].
\]
We can find a surjective ring homomorphism $\phi_n:$ $(\B^{{\otimes}(A+C)})_n \longrightarrow R_{\Gbar_n}$ which kills the second index of the label of every $y$'s, namely, for every $i,j$,
\[
\phi_n(x_{i,j}) = x_{i,j} 
\]
and
\[
\phi_n(y_{i,j}) = y_i
\]
In this case, it is not hard to see that the kernel is the ideal generated by differences of each pair of $y$'s within the same orbits, i.e. Ker$(\phi_n)$ = $(y_{i,k}-y_{i,j}$ | $1 \leq j < k \leq n, 1 \leq i \leq C)$. So $R_{\Gbar_n}$ is a quotient of $(\B^{{\otimes}(A+C)})_n$ by the first isomorphism theorem, making it a $(\B^{{\otimes}(A+C)})_n$-module by defining the action $r \cdot m = \phi_n(r)m$ for every $r \in (\B^{{\otimes}(A+C)})_n$ and $m \in R_{\Gbar_n}$.
\end{proof}

The next two results are based on the $\Gbar_\bullet$ defined above. 

\begin{lemma}\label{mainlemmaMnmodule}
Let $\M_n = R_{\Gbar_n}$ as defined above for each $n$, and for each injection $f: [k] \longrightarrow [l]$, define a $(\B^{{\otimes}(A+C)})_k$-module homomorphism $M_f: \M_k \longrightarrow \M_l$ such that for every $i,j$,
\[
M_f(x_{i,j}) = x_{i,f(j)}
\]
and
\[
M_f(y_{i}) = y_i
\]
Then $\M_\bullet$ is an $\FI$-module over $(\B^{{\otimes}(A+C)})_\bullet$. Furthermore, $\M_\bullet$ is finitely generated. 
\end{lemma}

\begin{proof}
We have shown in the previous lemma that each $R_{\Gbar_n}$ is equivalent to $(\B^{{\otimes}(A+C)})_n/I_n$ under the action by $(\B^{{\otimes}(A+C)})_n$ where $I_n$ = Ker$(\phi_n)$. Also, notice that each $M_f$ is equivalent to a transition map $\Tilde{X}_f$ between two quotient rings $(\B^{{\otimes}(A+C)})_k/I_k$ and $(\B^{{\otimes}(A+C)})_l/I_l$ such that for every $m \in (\B^{{\otimes}(A+C)})_k$,
\[
\Tilde{X}_f(m+I_k) = X_f(m)+I_l
\]
where $X_f$ represents the standard transition map that permutes variables within the same $\Sn_n$-orbit when treating $(\B^{{\otimes}(A+C)})_\bullet$ as an $\FI$-module over itself. Therefore, it suffices to show that $(\B^{{\otimes}(A+C)})_\bullet / I_\bullet$ is indeed an $\FI$-module over $(\B^{{\otimes}(A+C)})_\bullet$ with transition maps $\Tilde{X}_f$ defined above. 

We see that $\Tilde{X}_f$ is the same as $X_f$ except the fact that $\Tilde{X}_f$ are transition maps between quotient rings. Therefore, if the transition maps are well-defined, the second bullet point of the $\FI$-module $\M_\bullet$ in Definition \ref{FImoduleoverFIalgebradef} will be satisfied as follows:
\begin{align*} 
&\Tilde{X}_f(a \cdot (m+I_k)) \\
&= \Tilde{X}_f(am+I_k) \\
&= X_f(am)+I_l \\
&= (A_f(a) \cdot X_f(m))+I_l \\
&= A_f(a) \cdot (X_f(m)+I_l) \\
&= A_f(a) \cdot \Tilde{X}_f(m+I_k).
\end{align*}

Now it remains to show that $\Tilde{X}_f$ is a well-defined function for each injection $f$. 

We know that for each $n$, The ideal $I_n$ = Ker$(\phi_n)$ = $(y_{i,k}-y_{i,j}$ | $1 \leq j < k \leq n, 1 \leq i \leq C)$. It can be easily seen that when $k \leq l$, each generator of $I_k$ will be sent to a generator of $I_l$ by $X_f$ under any injection $f: [k] \longrightarrow [l]$. This means that $X_f(I_k) \subseteq I_l$ always holds, so for any $m_1 + I_k = m_2 + I_k$, $\Tilde{X}_f(m_1 + I_k) = \Tilde{X}_f(m_2 + I_k)$ always holds. So each $\Tilde{X}_f$ is indeed a well-defined $(\B^{{\otimes}(A+C)})_k$-module homomorphism. This shows that $(\B^{{\otimes}(A+C)})_\bullet / I_\bullet$ is an $\FI$-module over $(\B^{{\otimes}(A+C)})_\bullet$. 

As for finite generations, notice that $1+I_1 \in (\B^{{\otimes}(A+C)})_1 / I_1$ will be sent to every $(1+I_k) \in (\B^{{\otimes}(A+C)})_k / I_k$ under any injection from $[1]$ to $[k]$. So the element $1+I_1$ actually generates the entire $(\B^{{\otimes}(A+C)})_\bullet / I_\bullet$.
\end{proof}

The arguments above can be extended to the collection of quotient rings $(R/I)_{\Gbar_n} := R_{\Gbar_n}/I_{\Gbar_n}$.

\begin{theorem}\label{thmFImoduleovertensorB}
Let $\Mboldbar_\bullet$ be such that $\Mboldbar_n = (R/I)_{\Gbar_n}$ for every $n \in \N$, with the transition maps $\Mbar_f$ naturally extended from $M_f$ defined in the previous lemma. Then $\Mboldbar_\bullet$ is a finitely generated $\FI$-module over $(\B^{{\otimes}(A+C)})_\bullet$.
\end{theorem}

\begin{proof}
Since each $\Mboldbar_n$ is a quotient of $R_{\Gbar_n}$ and each $R_{\Gbar_n}$ is a $(\B^{{\otimes}(A+C)})_n$-module by Lemma \ref{lemmatensorBnmodule}, $\Mboldbar_n$ is also a $(\B^{{\otimes}(A+C)})_n$-module. Also, the collection of $\M_n = R_{\Gbar_n}$ forms an $\FI$-module over $(\B^{{\otimes}(A+C)})_\bullet$ by lemma \ref{mainlemmaMnmodule}, so the natural extension of transition maps preserves the second property in Definition \ref{FImoduleoverFIalgebradef}. 
Notice that $\Gbar_\bullet$ is an $\FI$-graph, so edges are preserved under any transition maps. In particular, for every $k \leq l$, the generators of $I_{\Gbar_k}$ are contained in $I_{\Gbar_l}$ under any transition maps from $\Gbar_k$ to $\Gbar_l$. Using the same argument in the lemma above, we conclude that all transition maps are well-defined. So $\Mboldbar_\bullet$ is an $\FI$-module over $(\B^{{\otimes}(A+C)})_\bullet$. Also, since $\Mboldbar_\bullet$ is overall a quotient of $\M_\bullet$, it should also be finitely generated.
\end{proof}

In Theorem \ref{thmdimR/I}, we related the independence number of a graph and the Krull dimension of the quotient ring. In fact, they can be further related to the Krull dimension of the modules $\Mbar_n$. Recall that given a ring $R$, the Krull dimension of an $R$-module $M$ is the Krull dimension of the quotient ring $R / \Ann_R(M)$. Notice that for the same ring, the Krull dimension might be different when the ring is treated as a ring or as an $R$-module. However, since $(R/I)_{\Gbar_n}$ is "similar" enough to $(\B^{{\otimes}(A+C)})_n$ (the first is a quotient of a quotient of the second), we can have the following result:

\begin{theorem}\label{Krulldimthm}
Let $\Mboldbar_\bullet$ be the $\FI$-module over $(\B^{{\otimes}(A+C)})_\bullet$ where $\Mboldbar_n = (R/I)_{\Gbar_n}$ for every $n \in \N$, as defined in the above theorem. Then for each $n$, the Krull dimension of $\Mboldbar_n$ equals to the Krull dimension of the quotient ring $(R/I)_{\Gbar_n}$.
\end{theorem}

\begin{proof}
For simplicity, denote $(\B^{{\otimes}(A+C)})_n$ as $S_n$ for now. It suffices to find an isomorphism between $S_n / \Ann_{S_n}\bigl((R/I)_{\Gbar_n}\bigr)$ and $(R/I)_{\Gbar_n}$. Using the definition of annihilator, 
\[
\Ann_{S_n}\bigl((R/I)_{\Gbar_n}\bigr) = \{s \in S_n \text{ } | \text{ } s \cdot m = I_{\Gbar_n} \text{ } \forall m \in (R/I)_{\Gbar_n}\}.
\]
Recall from the proof of Lemma \ref{lemmatensorBnmodule} that there is a surjective homomorphism $\phi_n$ between $S_n$ and $R_{\Gbar_n}$, and the action is defined by $r \cdot m = \phi_n(r)m$. Since $(R/I)_{\Gbar_n}$ is a quotient of $R_{\Gbar_n}$, this action can be naturally extended. Using this definition, we can find 

\begin{align*} 
&\Ann_{S_n}\bigl((R/I)_{\Gbar_n}\bigr) \\
&= \{s \in S_n \text{ } | \text{ } s \cdot m = I_{\Gbar_n} \text{ } \forall m \in (R/I)_{\Gbar_n}\} \\
&= \{s \in S_n \text{ } | \text{ } (\phi_n(s)+I_{\Gbar_n})(r+I_{\Gbar_n}) = I_{\Gbar_n} \text{ } \forall r+I_{\Gbar_n} \in (R/I)_{\Gbar_n}\} \\
&= \{s \in S_n \text{ } | \text{ } (\phi_n(s)r+I_{\Gbar_n}) = I_{\Gbar_n} \text{ } \forall \text{ } r+I_{\Gbar_n} \in (R/I)_{\Gbar_n}\} \\
&= \{s \in S_n \text{ } | \text{ } \phi_n(s)r \in I_{\Gbar_n} \text{ } \forall \text{ } r \in R_{\Gbar_n}\} \\
&= \{s \in S_n \text{ } | \text{ } \phi_n(s) \in I_{\Gbar_n}\}
\end{align*}

where the last equality holds because $I_{\Gbar_n}$ is an ideal of $R_{\Gbar_n}$, and $R_{\Gbar_n}$ is a ring with $1$. 

Now denote $\Ann$ as $\Ann_{S_n}\bigl((R/I)_{\Gbar_n}\bigr)$ for simplicity, and define a function $\psi_n: S_n/{\Ann} \longrightarrow (R/I)_{\Gbar_n}$ such that 

\[
\psi_n(s+\Ann) = \phi_n(s) + I_{\Gbar_n}.
\]

First check that this map is well-defined: suppose $s_1+\Ann = s_2+\Ann$, so $s_1 - s_2 \in \Ann$, so $\phi_n(s_1) - \phi_n(s_2) = \phi_n(s_1 - s_2) \in I_{\Gbar_n}$ by definition, so $\psi_n(s_1+\Ann) = \phi_n(s_1) + I_{\Gbar_n} = \phi_n(s_2) + I_{\Gbar_n} = \psi_n(s_2+\Ann)$. The properties for a surjective ring homomorphism are easy to check for $\psi_n$ given that $\phi_n$ is a surjective homomorphism. Finally, $\psi_n$ is injective because the kernel is trivial due to the fact that $s \in \Ann \iff \phi_n(s) \in I_{\Gbar_n}$.
\end{proof}

Now combining all the pieces together, we can show the linearity of independence number in a vertex-linear $\FI$-graph, which has been stated in Theorem \ref{mainthm}: 

\begin{theorem}
Let $G_\bullet$ be a vertex-linear $\FI$-graph, and denote $\alpha(G_n)$ as the independence number of each $G_n$. Then the growth of $\alpha(G_n)$ is eventually linear in $n$, i.e. there exists (integer) constants A, B such that for all $n \gg 0$, 
\[
\alpha(G_n) = An + B.
\]
\end{theorem}

\begin{proof}
Find the integer $d$ such that $G_\bullet$ vertex-linear stable degree $\leq d$. Let $\Gbar_\bullet$ be the new $\FI$-Graph defined in Definition 
\ref{defGbar}, and we can see that $\alpha(G_n) = \alpha(\Gbar_n)$ for every $n \geq d$. Theorem \ref{thmdimR/I} tells us that for each $n$, $\alpha(\Gbar_n) = \dim(R/I)_{\Gbar_n}$, and Theorem \ref{Krulldimthm} further tells us that $\dim(R/I)_{\Gbar_n} = \dim \Mboldbar_n$ where $\Mboldbar_\bullet$ is the finitely generated $\FI$-module over $(\B^{{\otimes}(A+C)})_\bullet$ constructed in Theorem \ref{thmFImoduleovertensorB}. Theorem \ref{linearkrulldimthm} tells us that there are constants $A$, $B$, and $u$ such that for every $n \geq u$, $\dim \Mboldbar_n = An + B$. This implies that $\alpha(\Gbar_n) = An + B$ for every $n \geq u$. So by taking $c = \max\{d, u\}$, we can conclude that for every $n \geq c$, $\alpha(G_n) = An+B$. 
\end{proof}

\section{The general case of Conjecture \ref{mainconj}} \label{generalcase}

\subsection{Some words on how our proof may generalize}

In the previous section, we proved linear growth for independence numbers of vertex linear $\FI$-graphs using entirely algebraic techniques. More specifically, we constructed an $\FI$-module over an $\FI$-algebra $(\B^{{\otimes}d})_\bullet$ whose number of variables grows linearly, and argued how the growth of the independence number relates with the grow of the Krull dimension. One may ask a question about whether a similar method can be used to explore $\alpha(G_n)$ for an $\FI$-graph whose vertex set $V(G_n)$ grows like a higher-degree polynomial, i.e. constructing an $\FI$-module over an $\FI$-algebra whose number of variables grows in the same degree. Unfortunately, the answer is not as clear; the proof of Theorem \ref{linearkrulldimthm} in the original paper relies on the fact that every finitely generated $\FI$-module $\M_\bullet$ over $(\B^{{\otimes}d})_\bullet$ is Noetherian, i.e. every $\FI$-submodule of $\M_\bullet$ is also finitely generated. However, this is not necessarily the case for $\FI$-modules over $\FI$-algebras whose number of variables grow faster than linear, as the following well-known ``folk-lore" example shows.

\begin{example}
Let $\A_\bullet$ be the $\FI$-algebra such that each $\A_n$ is a polynomial ring with variables labelled with pairs, i.e. $\A_n$ = $\Q[x_{i,j} \text{ } | \text{ } i, j \in [n], i \neq j]$ for each $n$. Notice that this $\FI$-algebra has quadratic growth in its number of variables. Then $A_\bullet$ can be considered as a finitely generated $\FI$-module over itself.

Now for each $n$, consider the ideal $I_n$ generated by all monomials that encode the edges of a cycle in the complete graph $K_n$. Some examples of $I_n$ for small $n$'s include:
\[
I_1 = (0); 
\]

\[
I_2 = ({x_{1,2}}^2);
\]

\[
I_3 = (x_{1,2}x_{2,3}x_{3,1}, {x_{1,2}}^2, {x_{1,3}}^2, {x_{2,3}}^2);
\]
\begin{align*}
I_4 = (&x_{1,2}x_{2,3}x_{3,4}x_{4,1}, x_{1,2}x_{2,4}x_{4,3}x_{3,1}, x_{1,3}x_{3,2}x_{2,4}x_{4,1}, x_{1,2}x_{2,3}x_{3,1}, x_{1,2}x_{2,4}x_{4,1}, \\
&x_{1,3}x_{3,4}x_{4,1}, x_{2,3}x_{3,4}x_{4,2}, {x_{1,2}}^2, {x_{1,3}}^2, {x_{1,4}}^2, {x_{2,3}}^2, {x_{2,4}}^2, {x_{3,4}}^2).
\end{align*}

For every $n \geq m$, we can see that $I_n$ contains the image of $I_m$ under every transition map (because the image of any element that is a multiple of a $k$-cycle will again be a multiple of a $k$-cycle), so $I_\bullet$ is indeed an $\FI$-submodule of $A_\bullet$. However, every $n$-cycle in $I_n$ cannot be generated by elements from the previous $I_m$, because an $n$-cycle does not contain an $m$-cycle for every $m < n$. So $I_\bullet$ is not finitely generated. This means that $A_\bullet$ is not Noetherian.
\end{example}

The above example indicates that we cannot use the "Noetherianity" property to obtain nice results for $\FI$-graphs whose number of vertices grows in degree higher than $1$. Furthermore, the independence number of an $\FI$-graph whose number of vertices grows in degree $k \geq 1$ need not grow like a polynomial in $n$, as discussed in the example of Johnson Graph right before Theorem \ref{mainthm}. However, it is worth noting that our conjecture \ref{mainconj} is true in this example. Moreover, our method to calculate the independence number relies on finding the Krull dimension of rings quotiented by a finitely generated square-free monomial ideal, which are the easiest type of ideals to deal with.

We will introduce our experimentation of verifying this conjecture in lower-degree cases later in this chapter. 

\subsection{The classification theorem for vertex-quadratic $\FI$-graphs}\label{quadclass}

Right before definition \ref{defVertexLinearFIGraph}, we mentioned that if $G_\bullet$ is an $\FI$-graph, then the size of $|V(G_n)|$ agrees with a polynomial with some degree $n$. In fact, this fact can be referred from theorem \ref{thmofsubgraph} by making $H = $ \{a single vertex\}. The structure of the $\FI$-graph gets more complicated for larger $n$, but there is still a manageable (but not necessarily neat) way to classify all $\FI$-graphs if $|V(G_n)|$ does not grow more than quadratically, i.e. $n \leq 2$. The idea of the proof follows the same guideline as the proof of classifying all edge-linear $\FI$-graphs in \cite[Theorem 3.13]{R}. 

Now we will define a classification graph for the vertex-quadratic $\FI$-graphs, which indeed allows loops and multi-edges. Notice that this definition is a bit cumbersome, but it indeed captures the eventual behaviour of every possible vertex-quadratic $\FI$-graph. The reader only needs to quickly skim the notation we define here; we will explain how we interpret these labels in the next definition.

\begin{definition} \label{classificationgraphdef}
A vertex-quadratic classification graph, denoted $\mathfrak{C}$, is a (finite) graph with each vertex labelled either $v^{(1)}, v^{(2)}, v^{(3)}$, or $v^{(4)}$, satisfying the followings:

\begin{enumerate}

    \item Every loop around $v^{(1)}$ will have one of the following five labels: $e_1^{(1)}, e_2^{(1)}, e_3^{(1)}, e_4^{(1)}, e_5^{(1)}$; \label{cond1}
    \item Every loop around $v^{(2)}$ will have one of the following two labels: $e_1^{(2)}, e_2^{(2)}$\label{cond2};
    \item At most one loop is allowed for each $v^{(3)}$, in which case this loop will be labelled $e_1^{(3)}$;
    \item No loop is allowed for each $v^{(4)}$;
    \item In the cases \ref{cond1} and \ref{cond2}, no two loops around the same vertex should have the same label; 
    \item Every edge between two $v^{(1)}$'s will have one of the following six labels: $e_1^{(1,1)}, e_2^{(1,1)}, e_3^{(1,1)}, e_4^{(1,1)}, $ $ e_5^{(1,1)}, e_6^{(1,1)}$, respectively. Among them, an edge labeled $e_2^{(1,1)}$ also carries the data of a source and target for the edge; \label{cond6}
    \item Every edge between one $v^{(1)}$ and one $v^{(2)}$ will have one of the following four labels: $e_1^{(1,2)}, e_2^{(1,2)}, \\ e_3^{(1,2)}, e_4^{(1,2)}$;
    \item Every edge between one $v^{(1)}$ and one $v^{(3)}$ will have one of the following three labels: $e_1^{(1,3)}, e_2^{(1,3)}, \\ e_3^{(1,3)}$;
    \item Every edge between two $v^{(2)}$'s will have one of the following three labels: $e_1^{(2,2)}, e_2^{(2,2)}, e_3^{(2,2)}$;
    \item Every edge between one $v^{(2)}$ and one $v^{(3)}$ will have one of the following two labels: $e_1^{(2,3)}, e_2^{(2,3)}$;
    \item At most one edge is allowed between two $v^{(3)}$'s, in which case this edge will be labelled $e_1^{(3,3)}$;
    \item Any edge connected to a vertex of type $v^{(4)}$ will carry one of the labels $e_1^{(1,4)}$, $e_1^{(2,4)}$, $e_1^{(3,4)}$, or $e_1^{(4,4)}$ depending on the vertex type on the other end;\label{cond12}
    \item In the cases \ref{cond6} - \ref{cond12}, no two edges between the same two vertices should have the same label, with the only exception that at most two edges labelled $e_2^{(1,1)}$ are allowed if they point in opposite directions. 
\end{enumerate}
\end{definition}

For the next definition, we will expand upon the meaning of each notation in the above definition. In order to prevent our exposition from becoming too long and clumsy, we only provide a full description for the vertices $v^{(x)}$ and the edges $e_j^{(x)}$ and $e_j^{(x,y)}$ where $x \geq 2$. It is our hope that the reader can fill in any missing details based on similar ideas to what we now consider.

\begin{definition}\label{defnotationofclassificationgraph}
Given a vertex quadratic classification graph as defined immediately above, with no vertices of type $v^{(1)}$, and an integer $n \geq 0$, we construct the \textbf{description map} $\mathfrak{f}(\mathfrak{C}, n)$ to be the graph obtained from $\mathfrak{C}$ and $n$ by interpreting each vertex and edge label of $\mathfrak{C}$, as follows: 

\begin{itemize}
    \item $v^{(2)}$ represents an \textbf{unordered-pair orbit} $\mathcal{O}_{\{a,b\}}$, which is full collection of vertices labelled by unordered pairs $\{\{a, b\}$ $|$ $a, b \in [n], a \neq b\}$ acted by $\Sn_n$;
    \item $v^{(3)}$ represents a \textbf{linear-growth orbit} $\mathcal{O}_{a}$, which is a full collection of vertices labelled by natural numbers $\{c$ $|$ $c \in [n]\}$ acted by $\Sn_n$;
    \item $v^{(4)}$ represents a \textbf{singleton orbit} $\mathcal{O}_{0}$, which is a single vertex invariant under $\Sn_n$-action;
    \item $e_1^{(2)}$ indicates that every two vertices in the corresponding $\mathcal{O}_{\{a,b\}}$ whose labels are disjoint, i.e. $\{m, n\}$ and $\{k, l\}$, are connected;
    \item $e_2^{(2)}$ indicates that every two vertices in the corresponding $\mathcal{O}_{\{a,b\}}$ whose labels share one common number, i.e. $\{m, n\}$ and $\{m, k\}$, are connected;
    \item $e_1^{(3)}$ indicates that every two vertices in the corresponding $\mathcal{O}_{a}$ are connected, i.e. such $\mathcal{O}_{a}$ forms a complete graph itself;
    \item For each $j$ between $1$ and $3$, $e_j^{(2,2)}$ indicates that every two vertices, one from each corresponding $\mathcal{O}_{\{a,b\}}$'s, are connected if their labels share $j-1$ common numbers, i.e. $\{m, n\}$ and $\{k, l\}$ for $e_1^{(2,2)}$, $\{m, n\}$ and $\{m, k\}$ for $e_2^{(2,2)}$, and $\{m, n\}$ and $\{m, n\}$ for $e_3^{(2,2)}$;
    \item For $j = 1$ or $2$, $e_j^{(2,3)}$ indicates that every two vertices, one from the corresponding $\mathcal{O}_{\{a,b\}}$ and the other from the corresponding $\mathcal{O}_{a}$, are connected if their labels share $j-1$ common numbers;
    \item For $j = 1$ or $2$, $e_j^{(3,3)}$ indicates that every two vertices, one from each corresponding $\mathcal{O}_{a}$'s, are connected if their labels share $j-1$ common numbers;
    \item $e_1^{(x,4)}$ indicates that the singleton vertex from the corresponding $\mathcal{O}_{0}$ is connected to every vertex in the other orbit (the other orbit can also be a singleton, in which case $x=4$).
    
\end{itemize}

\end{definition}

It can be seen that we have completely avoided classification graphs with vertices labeled $v^{(1)}$ in the above definition. As one might easily guess, the classification map should take these vertices to collections of vertices labeled by \emph{ordered} pairs in $[n]$. In other words, the above only fully describes how to classify vertex quadratic $\FI$-graphs whose quadratic term exclusively comes from orbits of \emph{unordered} pairs (which will be called \textbf{unordered-pair $\FI$-graphs} in later sections). The expositional difficulty that arises when one allows for ordered-pair orbits is that the number of possible edges is amplified by the fact that one needs to not only consider how much overlap the pair has but also where the overlap occurs in the pair. We believe that this makes the description significantly more complicated without providing further insights. All of the computer calculations described below, and indeed the code \cite{indcode} itself, only works with the unordered pair case. The reader can will be able to fill in the missing details above by imitating the ideas we will use in the proof of Theorem \ref{thmforclassificationgraph} 

Now, let $G = \mathfrak{f}(\mathfrak{C}, n)$ for fixed $\mathfrak{C}$ and $n$, and consider a relation $\sim_V$ on the set $V(G)$ such that 
\[
v_1 \sim_V v_2 \iff v_2 \in \mathcal{O}_{\Sn_n}(v_1).
\]

We can also analogously define the relation $\sim_E$ on $E(G)$. It is easy to check that $\sim_V$ and $\sim_E$ are both equivalence relations, so we can simplify by combining them into the equivalence relation $\sim$ on $V(G) \cup E(G)$. Now. if we consider the set of equivalence classes of $G$, namely $G / {\sim}$, we can see that it is exactly isomorphic to $\mathfrak{C}$. In fact, the graph $G$ is an extended version of $\mathfrak{C}$ but is naturally endowed with $\Sn_n$ action.

\begin{theorem} \label{thmforclassificationgraph}
Let $G_\bullet$ be a vertex-quadratic $\FI$-graph. Then there is a (unique) classification graph $\mathfrak{C}$ such that for any $n \gg 0$, $\mathfrak{f}(\mathfrak{C}, n) \cong G_n$. 
\end{theorem}

\begin{proof}
According to the classification theorem \cite[Theorem A]{RSW}, for every sufficiently large $n$, $V(G_n)$ is isomorphic to a disjoint union of some sets endowed with $\Sn_n$ action, i.e. $V(G_n) \cong \bigsqcup_i [X_i]_n$, and the growth of each $X_i$ agrees with a polynomial of $n$. Also, the theorem implies that the size of $X_i$:

\begin{itemize}
    \item grows quadratically if and only if it is isomorphic as an $\Sn_n$-set to the set of ordered or unordered pairs on $[n]$;
    \item grows linearly if and only if it is isomorphic as an $\Sn_n$-set to $[n]$;
    \item is a constant if and only if it is a single element, invariant under the $\Sn_n$ action.
\end{itemize}

In particular, for all $n \gg 0$,  the vertex set of $G_n$ is isomorphic as an $\Sn_n$-set to a disjoint union of copies of ordered-pair orbits, unordered-pair orbits, linear-growth orbits and singleton orbits (as defined in definition \ref{defnotationofclassificationgraph}).

To understand the full structure of $G_n$, with $n \gg 0$, it remains to identify the edges. By definition of an $\FI$-graph, we know that the edge set $E(G_n)$ is an $\Sn_n$-set in a way that extends the action on the vertex set. The question therefore becomes to classify the possible symmetric $\Sn_n$-relations between the four types of $\Sn$-sets that comprise the vertex set. It is easily checked that the possible such relations, excluding the cases of ordered pairs, are precisely those described in Definition \ref{defnotationofclassificationgraph}. The case the the possible $\Sn_n$-relations where one of the sets is a set of ordered pairs is left to the reader.

To complete the proof, we note that a second application of \cite[Theorem A]{RSW} to the edge set, instead of the vertex set, implies that the edge set must decompose as an $\Sn_n$-set in a consistent way when $n \gg 0$. When both the edge and vertex sets have stabilized in this sense, it must be the case that the graph $G_n$ is isomorphic to $\mathfrak{f}(\mathfrak{C},n)$.
\end{proof}

\begin{remark}
Although not being directly implemented in our later work, it is worth mentioning why we make $e_2^{(1, 1)}$ to be directed edges in definition \ref{classificationgraphdef} if the reader wants to further fill the rest of the details for definition \ref{defnotationofclassificationgraph}. We add an edge labelled $e_2^{(1, 1)}$ between two vertices labelled $v^{(1)}$ to indicate that a vertex $v_1$ in the "first" $\mathcal{O}_{(a,b)}$ is connected to a vertex $v_2$ in the "second" $\mathcal{O}_{(a,b)}$ if their labels share \emph{exactly one} common number $m$, that appears in the first index of $v_1$'s label and the second of $v_2$'s, i.e. $(m, n)$ and $(k, m)$, with $n \neq k$. In this case, we have to use arrows to indicate which $\mathcal{O}_{(a,b)}$ is the "first" and which is the "second" due to the antisymmetry of this case. In fact, this is the only case we need to consider which orbit comes "first" in a vertex-quadratic $\FI$-graph; all the remaining cases between two same-type orbits are symmetric, and cases between two different-type orbits will not cause confusions.
\end{remark}

A few examples of classification graphs and the corresponding $\FI$-graphs are provided here for the reader to get more familiar with our descriptions above. 

\begin{example}\label{classificationgraphexample}
Let $G_n = KG_{n,2}$ for each $n$ be a collection of Kneser graphs. Then $G_n \cong \mathfrak{f}(\mathfrak{C}, n)$ for every $n \geq 2$, where the classification graph $\mathfrak{C}$ should be a graph with only one vertex $v^{(2)}$ with a loop $e_1^{(2)}$ around it. Similarly, if $G_n = J_{n,2}$ is a collection of Johnson graphs, then $\mathfrak{C}$ should be a graph with only one vertex $v^{(2)}$ with a loop $e_2^{(2)}$ around it.
\end{example}

\begin{example}
Let $G_n = K_{n,n}$ for each $n$ be a collection of complete bipartite graphs. Then $G_n \cong \mathfrak{f}(\mathfrak{C}, n)$ for every $n \geq 1$, where $\mathfrak{C}$ should contain two $v^{(3)}$ vertices, and they are connected by two edges, $e_1^{(3,3)}$ and $e_2^{(3,3)}$. This is because connecting all vertices between two copies of $[n]$ is the same as connecting vertices between two copies of $[n]$ if their labels share zero or one common number. If we wanted to alter our example to turn one of the vertex partitions into a complete graph, then there is another loop $e_1^{(3)}$ around the corresponding $v^{(3)}$. 
\end{example}

\begin{example}
Let $\mathfrak{C}$ be a graph satisfying the following: 
\begin{itemize}
    \item it has three vertices, two of them are $v^{(2)}$'s and the remaining one is $v^{(4)}$;
    \item each $v^{(2)}$ has two loops $e_1^{(2)}$ and $e_2^{(2)}$; and
    \item the two $v^{(2)}$'s are connected by $e_1^{(2,2)}$, $e_2^{(2,2)}$ and $e_3^{(2,2)}$, and each $v^{(2)}$ is connected to the $v^{(4)}$ by $e_1^{(2,4)}$. 
\end{itemize}
Then it is not hard to see that every two vertices in $G_n = \mathfrak{f}(\mathfrak{C},n)$ are adjacent for each $n \geq 2$; in fact, each $G_n$ forms a complete graph $K_{\binom{n}{2}+1}$. Importantly, however, the $\Sn_n$ action on this complete graph is not transitive. In fact, it has a vertex that is invariant under the action of $\Sn_n$, as well as two orbits, each of which isomorphic to unordered pairs on $[n]$.

\end{example}

\subsection{Sage experimentation}

In order to support our conjecture \ref{mainconj}, we created a Sagemath program that can compute the independence number $\alpha(G_n)$ for each not-too-large $n$ of any unordered-pair $\FI$-graph generated from a classification graph (see the description right after the definition \ref{defnotationofclassificationgraph}.) The underlying idea is that users can manually input a classification graph, and the program outputs the corresponding $\FI$-graph and displays $G_n$ and its independence number for each sufficiently small $n$. Those independence numbers $\alpha(G_n)$ are displayed together so that users can easily see that there is a quasi-polynomial pattern in this sequence of numbers. This program can also generate vertex-linear $\FI$-graphs by not including any unordered-pair orbits, and the user can verify that all such $\FI$-graphs satisfy Theorem \ref{mainthm}. There is also a method to mass-generate random unordered-pair $\FI$-graphs and compute the corresponding independence numbers. 

Due to the data structure of the "graph" object in Sagemath, the classification graphs we defined in definition \ref{defnotationofclassificationgraph} cannot be directly inputted. The user guide gives instructions about how to manually convert the classification graph into a form that is usable in the code. Further details, as well as the user guide, are included in the program itself \cite{indcode}. 

\begin{remark}
Since our program generates the entire $\FI$-graph from the input classification graph, this means that every $G_n$ where $n \geq 2$ is an extended version of a classification graph, $\mathfrak{f}(\mathfrak{C}, n)$. In particular, our program does not have a mechanism for generating examples that have odd behaviors for small $n$. However, since the $G_n$'s of every $\FI$-graph \emph{eventually} look like an extended version of a fixed classification graph according to theorem \ref{thmforclassificationgraph}, the patterns observed from the $\FI$-graph we generate in the program can indeed generalize to every possible unordered-pair $\FI$-graph. Allowing outliers for small $n$'s does not add any insights to capturing the patterns of independence numbers, since we are ultimately interested in understanding things asymptotically in $n$. 
\end{remark}

\subsection{Some trends apparent from our experimentation}

By repeatedly running the random $\FI$-graph generator in our Sagemath program, we observe the following patterns for any $G_\bullet$ that is vertex-quadratic where the second degree entirely comes from unordered-pair orbits.

\begin{enumerate}
    \item For every $n \gg 0$, the independence number $\alpha(G_n)$ agrees with an (at most quadratic-degree) quasi-polynomial of \textbf{at most two pieces}; in other words, 
    \[
    \alpha(G_n) = \left\{
    \begin{array}{ll}
          f_1(n), & n \text{ is odd} \\
          f_2(n), & n \text{ is even}
    \end{array} 
    \right.
    \]
    where $f_1$, $f_2$ are polynomials of degree at most two. This observation shows that conjecture \ref{mainconj} is likely to be true beyond vertex-linearity. \label{condition1trend}
    \item If the classification graph $\mathfrak{C}$ corresponding to $G_\bullet$ \textbf{does not} contain an orbit of the Johnson graph (meaning that a vertex $v^{(2)}$ with exactly a loop $e_2^{(2)}$ around is \textbf{not} an induced subgraph of $\mathfrak{C}$; see Example \ref{classificationgraphexample}), then $\alpha(G_n)$ indeed agrees with a polynomial $f$ for sufficiently large $n$ (or in other words, $f_1 = f_2$). \label{condition2trend}
    \item The (quasi)-polynomial in \ref{condition1trend} and \ref{condition2trend} has degree two if and only if it contains an unordered-pair orbit without edges, i.e. a classification graph containing a vertex $v^{(2)}$ without any loop. \label{condition3trend}
\end{enumerate}

\begin{remark}
If we assume conjecture \ref{mainconj} to be true, then we can talk about the \textbf{period} of the quasi-polynomial $\alpha(G_n)$, or equivalently, the smallest number of subsets such that 
\[
\{\alpha(G_n) \text{ } | \text{ } n \gg 0\} = \bigsqcup X_i
\]
where the collection of elements in each $X_i$ agrees with a unique polynomial in $n$. For any $G_\bullet$ from our experimentation, the period of $\alpha(G_n)$ will be either $1$ or $2$ according to observation \ref{condition1trend} above. One possible direction of future studies of $\FI$-graphs along with proving conjecture \ref{mainconj} is to understand the relationship between the polynomial-degree of $|V(G_n)|$ and the period of $\alpha(G_n)$. Based on the current patterns we found, it is plausible to conjecture that for any $\FI$-graph $G_\bullet$, 
\[
T \text{ } | \text{ } d!
\]
where $T$ is the period of $\alpha(G_n)$ and $d$ is the degree of $|V(G_n)|$.

\end{remark}

\begin{remark}
The forward direction of observation \ref{condition3trend} is not true in a general vertex-quadratic $\FI$-graph where ordered-pair orbits are allowed. Let $\mathfrak{C}$ be the classification graph with a single vertex $v^{(1)}$ and a loop to indicate that in any $\mathfrak{f}(\mathfrak{C}, n)$ (recall this notation from definition \ref{defnotationofclassificationgraph}), any vertex labelled $(m,l)$ is connected to the one labelled $(l,k)$ as long as $m \neq l \neq k$. This $\FI$-graph turns out to be eventually connected, while for any $n \geq 2$, we can find an independent set $\bigl\{(a,b)$ | $1 \leq a \leq \lfloor \frac{n}{2} \rfloor, \lfloor \frac{n}{2} \rfloor + 1 \leq b \leq 2 \cdot \lfloor \frac{n}{2} \rfloor\bigr\}$ whose size is $\lfloor \frac{n}{2} \rfloor^2$, concluding that the independence number grows quadratically. 
\end{remark}

\begin{remark}
The converse statement of observation \ref{condition2trend} is not necessarily true. Let $G_n = \mathfrak{f}(\mathfrak{C}, n)$ for each $n$, where the classification graph $\mathfrak{C}$ is a graph satisfying the following:
\begin{itemize}
    \item it has three vertices, two of them are $v^{(2)}$ and the remaining one is $v^{(3)}$;
    \item each $v^{(2)}$ has a loop $e_2^{(2)}$, and the $v^{(3)}$ has a loop $e_1^{(3)}$; and
    \item one $v^{(2)}$ is connected to $v^{(3)}$ by $e_2^{(2,3)}$.
\end{itemize}
It is not hard to see that each $G_n$ is a disjoint union of $J_{n,2}$ and $J_{n+1,2}$ (in fact, $J_{n+1,2}$ belongs to the collection of Johnson graphs with $\FI$-shift applied once \cite[Definition 2.7]{CEFN}), making the independence number $\alpha(G_n) = \lfloor \frac{n}{2} \rfloor + \lfloor \frac{n+1}{2} \rfloor = n$ for every $n \geq 2$, which is exactly a polynomial. 
\end{remark}

Due to the high complexity of computing independence number of a graph as well as some other limitations, our Sagemath program can only handle the calculations up to approximately $G_{15}$ based on classification graphs that are simple enough. It is a natural question to ask whether there are $\FI$-graphs whose $\alpha(G_n)$ eventually grow like (quasi)-polynomials but the pattern does not emerge until very large $n$, so the pattern cannot be detected by the code. This inspires the following definition that is analogous to the definition of "vertex-linear stable degree". 

\begin{definition}
Let $G_\bullet$ be any $\FI$-graph. We say $G_\bullet$ has \textbf{independence-number stable degree} $\leq d$ if these exists a minimal integer $d$ such that $\alpha(G_n)$ agree with a (quasi)-polynomial for every $n \geq d$. 
\end{definition}

There is not a clear pattern of independence-number stable degree among different examples of $\FI$-graphs. However, it turns out that for every integer $m>0$, we can come up with examples of $\FI$-graph $G_\bullet$ where $|V(G_n)|$ eventually agrees with a polynomial of degree $m$, and $G_\bullet$ has arbitrarily large independence-number stable degree. For the following examples, we take $G_n = \mathfrak{f}(\mathfrak{C}, n)$ based on a classification graph $\mathfrak{C}$, and $k \gg 0$ is a large positive integer.

\begin{example}
Let $\mathfrak{C}$ be a graph satisfying the following:
\begin{itemize}
    \item it has $k+1$ vertices, $k$ of them are $v^{(4)}$ and the remaining one is $v^{(3)}$;
    \item each $v^{(4)}$ is connected to $v^{(3)}$ by an edge $e_1^{(3,4)}$.
\end{itemize}
It can be easily deduced that 
\[
\alpha(G_n) = \left\{
\begin{array}{ll}
    k, & n < k \\
    n, & n \geq k
    \end{array} 
\right.
\]
which means that $G_\bullet$ has independence-number stable degree $\leq k$.
\end{example}

We can come up with similar examples as above by setting up large number of disjoint singletons and connect them to all other orbits, imposing that these singleton vertices form the largest independent set of $G_n$ until $n$ gets sufficiently large. This construction is relatively uninteresting as it does not reveal much structure of $\FI$-graphs. Some examples that are less ad hoc are below.

\begin{example}\label{groupofKnexample}
Let $\mathfrak{C}$ be a graph satisfying the following:
\begin{itemize}
    \item it has $k$ vertices, all of them are $v^{(3)}$;
    \item each $v^{(3)}$ has a loop $e_1^{(3)}$;
    \item every two $v^{(3)}$ are connected by an edge $e_2^{(3,3)}$.
\end{itemize}
The $\FI$-graph associated to $\mathfrak{C}$ is $k$ copies of the complete graph $K_n$, such that between any two such complete graphs, vertices are only connected if they have the same label. In this example, observe that in any independent set:
\begin{itemize}
    \item for each integer $j$, there is at most one vertex labelled $j$; and 
    \item there is at most one vertex associated to any given orbit. 
\end{itemize}
Therefore, we can deduce that 
\[
\alpha(G_n) = \left\{
\begin{array}{ll}
    n, & n < k \\
    k, & n \geq k
    \end{array} 
\right.
\]
which means that $G_\bullet$ has independence-number stable degree $\leq k$.
\end{example}

This remaining example is a vertex-quadratic $\FI$-graph analogous to the example above while showing some interesting properties of the Kneser graph $KG_{n,2}$. 

\begin{example}\label{groupofKGn2example}
Let $\mathfrak{C}$ be a graph satisfying the following:
\begin{itemize}
    \item it has $k$ vertices, all of them are $v^{(2)}$;
    \item each $v^{(2)}$ has a loop $e_1^{(2)}$;
    \item every two $v^{(2)}$ are connected by two edges $e_2^{(2,2)}$ and $e_3^{(2,2)}$.
\end{itemize}
Then it turns out that 
\[
\alpha(G_n) = \left\{
\begin{array}{ll}
    n, & n \leq 3k \text{ and $n$ is a multiple of $3$ }  \\
    n-1, & \text{else}
    \end{array} 
\right.
\]
which means that $G_\bullet$ has independence-number stable degree $\leq 3k+1$.
\end{example}

We can see that the independence-number stable degree of example \ref{groupofKGn2example} also grows arbitrarily large as $k$ grows. Before we present the details, notice that the complete graph $K_n$ is exactly the Kneser graph $KG_{n,1}$. Therefore, the $\FI$-graph in this example, as well as example \ref{groupofKnexample}, can be viewed as a union of $k$-copies of Kneser graphs where every two vertices from different orbits are connected as long as their labels are not disjoint. One may then ask whether the independence-number stable degree always grow arbitrarily large for multiple copies of Kneser graphs $KG_{\bullet,r}$ with this property for larger $r$, given that now it holds for $r = 1$ and $2$. We claim that this is not the case for any $r \geq 3$.

\textbf{For the following definitions and lemmas, let $G_\bullet$ be an $\FI$-graph of $k$-copies of Kneser graphs $KG_{n,r}$, connected to each other by adding edges between subsets that intersect non-trivially, and assume $r \geq 2$.} 

We will now work towards proving the following.

\begin{proposition}\label{techexample}
Suppose $r \geq 3$. Then for every $n \geq 2r$, the independence number is given by $\alpha(G_n) = \binom{n-1}{r-1}$. 
\end{proposition}

\begin{definition}
for any $n \geq r$, let $H$ be an independent set of $G_n$, and let $\mathcal{O}$ be a $KG_{\bullet,r}$-orbit. We denote $H_\mathcal{O} = H \cap [\mathcal{O}]_n$ to be the independent subset restricted to the Kneser graph $KG_{n,r}$ corresponding to this orbit. 
\end{definition}

Denote $\mathcal{O}_1$, $\mathcal{O}_2$, ..., $\mathcal{O}_k$ to be the Kneser-graph orbits of $G_\bullet$. Notice that our construction of $G_\bullet$ implies that if $H$ is an independent set of $G_n$ and $\mathcal{O}_i$, $\mathcal{O}_j$ are different Kneser-graph orbits, then there does not exist a vertex $v_1 \in H_{\mathcal{O}_i}$ labelled $\{a_1, ..., a_r\}$ and a vertex $v_2 \in H_{\mathcal{O}_j}$ labelled $\{b_1, ..., b_r\}$ such that these two labels are not disjoint. This means that there exists a subset $S \subseteq [n]$, and a partition $S_1, ..., S_k \subseteq S$, such that the numbers in the label of any vertex in $H_{\mathcal{O}_i}$ must exclusively come from $S_i$. In this case, each $S_i$ is either empty or has size not smaller than $r$. Clearly the set $S$ and the partition $S_1, ..., S_k$ are uniquely determined.

Now, notice that for each $i$, there exists a largest independent set of $G_n \cap [\mathcal{O}_i]_n$ whose vertices can only be labelled using elements from $S_i$. Denote it $B_i$. It follows from our construction of $G_\bullet$ that $B = \bigcup_{i=1}^k B_i$ is the (not necessarily unique) largest independent set of $G_n$ under this given partition $S_1, ..., S_k$. Since $H$ is an independent set under the same partition, $B$ is a larger (or at the very least, not smaller) independent set than $H$. Since our goal is to find the largest independent set of $G_n$, so WLOG we only need to compare the largest independent set of every possible partition of subsets of $[n]$. 

We can further reduce this problem by eliminating partitions of proper subsets of $[n]$. Assume $S_1, ..., S_k$ is a partition of $S$ where $S$ is a proper subset of $[n]$, and $H$ is a largest independent set associate to this partition. Then we can add the missing numbers into any nonzero $S_i$ to form a partition of $[n]$, and clearly $H$ is an independent set under this new partition. This means that if $H'$ is a largest independent set associate to this new partition, then $|H'| \geq |H|$. Therefore, WLOG we only need to compare the largest independent set of every possible partition of $[n]$.

\begin{lemma}\label{sizeofsetlemma}
Let $S_1, ..., S_k$ be a partition of a subset of $[n]$, and $H$ be a largest independent set associate to this partition. Denote $m_i$ to be the size of $S_i$. Then we have
\[
|H| = \sum_{i=1}^k |H_{\mathcal{O}_i}|
\]
where 
\[
|H_{\mathcal{O}_i}| = \left\{
\begin{array}{ll}
    0, & m_i = 0 \\
    \binom{m_i}{r}, & r \leq m_i \leq 2r-1 \\
    \binom{m_i-1}{r-1}, & m_i \geq 2r.
    \end{array} 
\right.
\]
\end{lemma}

\begin{proof}
If $K$ is a subset of $[n]$, and denote $V_K \subseteq V(KG_{n,r})$ to be the collection of all vertices whose labels only contain elements from $K$, then it is easy to see that the induced subgraph $KG_{n,r}[V_K]$ is itself a Kneser graph with size of vertex set $|V_K| = \binom{|K|}{r}$. Since $H_{\mathcal{O}_i}$ is the largest independent set of each $\big[KG_{n,r}[V_{S_i}]\big]_i$, so the result above is a direct consequence of Erd\"os-Ko-Rado Theorem.
\end{proof}

We present the following three elementary results for later use:

\begin{lemma}\label{binomiallemma}
Let $r \leq a \leq 2r-2$ and $b \geq 2r-1$. Then $\binom{a}{r} + \binom{b}{r-1} < \binom{a+b-1}{r-1}$.
\end{lemma}

\begin{proof}
Notice that $a \geq 2$, $a < b$, $a-r \leq r-2$ and $b > 2(r-2)$. So 
\begin{align*}
\binom{a}{r} + \binom{b}{r-1} &= \binom{a}{a-r} + \binom{b}{r-1} \\ 
&< \binom{b}{a-r} + \binom{b}{r-1} \\
&\leq \binom{b}{r-2} + \binom{b}{r-1} \\
&= \binom{b+1}{r-1} \leq \binom{a+b-1}{r-1}.
\end{align*}
\end{proof}

\begin{lemma}\label{easybinomiallemma}
Let $r \geq 3$. Then $2 \cdot \binom{2r-1}{r-1} < \binom{4r-3}{r-1}$. 
\end{lemma}

\begin{proof}
Exercise. Notice that $4r-4 = 2 \cdot (2r-2)$.
\end{proof}

\begin{lemma}\label{binomiallemma2}
Let $r \geq 3$, $c, d \geq 2r-1$. Then $\binom{c}{r-1} + \binom{d}{r-1} < \binom{c+d-1}{r-1}$. 
\end{lemma}

\begin{proof}
\begin{align*}
\binom{c}{r-1} + \binom{d}{r-1} &= \Bigg(\binom{2r-1}{r-1} + \sum_{i=2r-1}^{c-1} \binom{i}{r-2}\Bigg) + \Bigg(\binom{2r-1}{r-1} + \sum_{j=2r-1}^{d-1} \binom{j}{r-2}\Bigg) \\
&= 2 \cdot \binom{2r-1}{r-1} + \sum_{i=2r-1}^{c-1} \binom{i}{r-2} + \sum_{j=2r-1}^{d-1} \binom{j}{r-2} \\
&< \binom{4r-3}{r-1} + \sum_{i=2r-1}^{c-1} \binom{i}{r-2} + \sum_{j=2r-1}^{d-1} \binom{j}{r-2} \text{ (by lemma \ref{easybinomiallemma})}\\
&\leq \binom{4r-3}{r-1} + \sum_{i=(2r-1)+(2r-2)}^{(c-1)+(2r-2)} \binom{i}{r-2} + \sum_{j=(2r-1)+(c-1)}^{(d-1)+(c-1)} \binom{j}{r-2} \\
&= \binom{4r-3}{r-1} + \sum_{i=4r-3}^{c+2r-3} \binom{i}{r-2} + \sum_{j=c+2r-2}^{c+d-2} \binom{j}{r-2} \\
&= \binom{c+d-1}{r-1}. 
\end{align*}
\end{proof}

For the remaining lemmas, let $n$ be an integer $\geq 2r$.

\begin{lemma}\label{lemmaeliminatesmallpartition}
Let $S_1, ..., S_k$ be a partition of $[n]$ with at least two nonempty $S_i$'s, and let $H$ be a largest independent set of $G_n$ associate to this partition. Assume $r \leq |S_i| \leq 2r-2$ for some $i$. If we treat $S_i$ and another nonempty set as a whole to form a new partition, and let $H'$ be a largest independent set of this new partition, then $|H'| \geq |H|$.
\end{lemma}

\begin{proof}
Let $S_i$ and $S_j$ be two nonempty sets of this partition, and let $r \leq |S_i| \leq 2r-2$. Now, set $S'_i = S_i \cup S_j$ and $S'_j = \emptyset$. Then $S_1, ..., S_k$ by replacing $S_i$ with $S'_i$ and $S_j$ with $S'_j$ is again a partition of the same set. Let $H'$ be the largest independent set of this new partition, and we can see the only difference between $H$ and $H'$ is that the independent subset $H_{\mathcal{O}_i} \cup H_{\mathcal{O}_j}$ is replaced by $H'_{\mathcal{O}_i}$. So suffice to compare $|H_{\mathcal{O}_i} \cup H_{\mathcal{O}_j}|$ and $|H'_{\mathcal{O}_i}|$. Now, denote $|S_i| = a$ (so $r \leq a \leq 2r-2$) and $|S_j| = b$, so $|S'_i| = a+b$. Since $a, b \geq r$, so $a+b \geq 2r$. Then lemma \ref{sizeofsetlemma} implies that $|H_{\mathcal{O}_i}| = \binom{a}{r}$ and $|H'_{\mathcal{O}_i}| = \binom{a+b-1}{r-1}$.

Case 1: $r \leq b \leq 2r-2$, so $|H_{\mathcal{O}_j}| = \binom{b}{r}$ by lemma \ref{sizeofsetlemma}. Then 
\begin{align*}
\binom{a}{r} + \binom{b}{r} &\leq \binom{2r-2}{r} + \binom{2r-2}{r} \\
&= \binom{2r-2}{r-2} + \binom{2r-2}{r-2} \\
&< \binom{2r-2}{r-2} + \binom{2r-2}{r-1} \\
&= \binom{2r-1}{r-1} \\
&\leq \binom{a+b-1}{r-1}.
\end{align*}

Case 2: $b = 2r-1$, so again $|H_{\mathcal{O}_j}| = \binom{b}{r}$ by lemma 
\ref{sizeofsetlemma}. Then $\binom{a}{r} + \binom{b}{r}$ = $\binom{a}{r} + \binom{b}{r-1} < \binom{a+b-1}{r-1}$ by lemma \ref{binomiallemma}.

Case 3: $b \geq 2r$, so $|H_{\mathcal{O}_j}| = \binom{b-1}{r-1}$ by lemma \ref{sizeofsetlemma}. Then $\binom{a}{r} + \binom{b-1}{r-1} < \binom{a}{r} + \binom{b}{r-1} < \binom{a+b-1}{r-1}$ by lemma \ref{binomiallemma}.

Combining the cases above, we conclude that $|H_{\mathcal{O}_i} \cup H_{\mathcal{O}_j}| < |H'_{\mathcal{O}_i}|$, so $H'$ is a strictly larger independent set of $G_n$ than $H$ is. We can say $|H'| \geq |H|$ for more generality.
\end{proof}

Lemma \ref{lemmaeliminatesmallpartition} implies that we can find a largest independent set $H \subseteq V(G_n)$ such that each set $S_i$ in the associated partition $S_1, ..., S_k$ determined by $H$ must be either empty or contain at least $2r-1$ elements. 

\begin{lemma}\label{lemmaeliminatelargepartition}
Let $S_1, ..., S_k$ be a partition of $[n]$ with at least two nonempty $S_i$'s and $|S_i| \geq 2r-1$ for each $i$. Let $H$ be a largest independent set of $G_n$ associate to this partition. Assume $|S_i| \geq 2r$ for some $i$. If we treat $S_i$ and another nonempty set as a whole to form a new partition, and let $H'$ be a largest independent set of this new partition, then $|H'| \geq |H|$.
\end{lemma}

\begin{proof}
Let $S_i$ and $S_j$ be two nonempty sets of this partition, and let $|S_i| = a \geq 2r$. We adopt the identical method in lemma \ref{lemmaeliminatesmallpartition}, where this time $|H_{\mathcal{O}_i}| = \binom{a-1}{r-1}$ and $|H'_{\mathcal{O}_i}| = \binom{a+b-1}{r-1}$ by lemma \ref{sizeofsetlemma}. Recall that $b \geq 2r-1$. 

First, assume $r \geq 3$. 

Case 1.1: $b = 2r-1$, so $|H_{\mathcal{O}_j}| = \binom{b}{r} = \binom{b}{r-1}$ by lemma \ref{sizeofsetlemma}. Then $\binom{a-1}{r-1} + \binom{b}{r-1} < \binom{a+b-2}{r-1} < \binom{a+b-1}{r-1}$, where the first inequality comes from lemma \ref{binomiallemma2}.

Case 1.2: $b \geq 2r$, so $|H_{\mathcal{O}_j}| = \binom{b-1}{r-1}$ by lemma \ref{sizeofsetlemma}. Then $\binom{a-1}{r-1} + \binom{b-1}{r-1} < \binom{a+b-3}{r-1} < \binom{a+b-1}{r-1}$, where the first inequality comes from lemma \ref{binomiallemma2}.

Now, assume $r=2$.

Case 2.1: $b = 2r-1 = 3$, so $|H_{\mathcal{O}_j}| = \binom{b}{r} = \binom{b}{r-1} = b$ by lemma \ref{sizeofsetlemma}. Then $\binom{a-1}{r-1} + \binom{b}{r} = \binom{a-1}{1} + b = (a-1) + b = a+b-1 \leq \binom{a+b-1}{r-1}$.

Case 2.2: $b \geq 2r$, so $|H_{\mathcal{O}_j}| = \binom{b-1}{r-1} = b-1$ by lemma \ref{sizeofsetlemma}. Then $\binom{a-1}{r-1} + \binom{b-1}{r-1} = (a-1) + (b-1) = a+b-2 < a+b-1 \leq \binom{a+b-1}{r-1}$.

Combining the cases above, we conclude that $|H_{\mathcal{O}_i} \cup H_{\mathcal{O}_j}| \leq |H'_{\mathcal{O}_i}|$, so $H'$ is a independent set of $G_n$ whose size is larger than or equals to that of $H$, i.e. $|H'| \geq |H|$.
\end{proof}

This lemma combined with lemma \ref{lemmaeliminatesmallpartition} guarantee that when we look for a largest independent set, it is never required to partition $[n]$ into multiple nonempty sets unless each of them has size $2r-1$. 

\begin{corollary}\label{conditionindependentsetcorollary}
There exists a largest independent set $H \subseteq V(G_n)$ such that at least one of the followings holds:
\begin{enumerate}
    \item There is only one nonempty set among the corresponding partition $S_1, ..., S_k$, whose size is $n$; or \label{condition1corollary}
    \item The size of each set $S_i$ is either $0$ or $2r-1$. \label{condition2corollary}
\end{enumerate}
\end{corollary}

\begin{proof}
This is just a rephrase of the sentence right before this corollary, which follows lemma \ref{lemmaeliminatesmallpartition} and lemma \ref{lemmaeliminatelargepartition}.
\end{proof}

\begin{corollary}\label{notmultiplecorollary}
If $n$ is not a multiple of $2r-1$, then the independence number $\alpha(G_n) = \binom{n-1}{r-1}$. 
\end{corollary}

\begin{proof}
Since $n$ is not a multiple of $2r-1$, clearly no independent set of $G_n$ satisfies condition \ref{condition2corollary} of corollary \ref{conditionindependentsetcorollary}. So there must be a largest independent set $H \subseteq V(G_n)$ that satisfies condition \ref{condition1corollary} of corollary \ref{conditionindependentsetcorollary}, say $|S_i| = n$. (Note: this is equivalent to say that all vertices of $H$ must come from a single Kneser graph $KG_{n,r}$.) Since $n \geq 2r$, so $\alpha(G_n) = |H| = |H_{\mathcal{O}_i}| = \binom{n-1}{r-1}$ by lemma \ref{sizeofsetlemma}.
\end{proof}

\begin{proof}[Proof of Proposition \ref{techexample}]
Lemma \ref{notmultiplecorollary} discusses the case such that $n$ is not a multiple of $2r-1$, so now it remains to discuss the case where $n$ is a multiple of $2r-1$, say $n = c \cdot (2r-1)$ where $c >1$. Clearly, the largest independent set that satisfies condition \ref{condition2corollary} of corollary \ref{conditionindependentsetcorollary} will have size $\binom{n-1}{r-1} = \binom{c \cdot (2r-1)-1}{r-1}$. Also, the largest independent set that satisfies condition \ref{condition1corollary} of corollary \ref{conditionindependentsetcorollary} (it does not even exist if $c > k$ where $k$ is the number of copies of Kneser graphs in $G_n$) has size $c \cdot \binom{2r-1}{r-1}$. It is easy to verify that $c \cdot \binom{2r-1}{r-1} < \binom{c \cdot (2r-1)-1}{r-1}$ by using lemma \ref{easybinomiallemma} and \ref{binomiallemma2} in an inductive argument. So $\alpha(G_n) = \binom{n-1}{r-1}$. 
\end{proof}

The above proposition shows that if $G_\bullet$ is $k$-copies of Kneser graphs $KG_{\bullet,r}$ where every two vertices from different Kneser-graph orbits are connected if their labels are not disjoint, then for every $k \geq 3$, the independence-number stable degree is $\leq 2r$ (in particular, $=2r$) no matter what $k$ is. The claim discussed by the end of example \ref{groupofKGn2example} is hereby supported. 

This example illustrates that the independence number stable range can behave in somewhat unexpected ways across examples. It remains an interesting problem to determine what combinatorial or algebraic property of the $\FI$-graph is dictating this stable range. In the case of $\FI$-modules, stable ranges can be seen as being dictated by stability degrees \cite{CEF} (the combinatorial explanation) or by the local cohomology theory \cite{LR} (the algebraic explanation). It would be interesting to see similar analyses here.

\begin{remark}
Example \ref{groupofKGn2example} discusses when $r=2$, which is special mainly because lemmas \ref{easybinomiallemma} and \ref{binomiallemma2} do not hold for $r=2$. However, all the pieces without specifying $r \geq 3$ holds for $r=2$, so one can easily obtain our results in this example using these pieces.
\end{remark}

\printbibliography

\end{document}